\documentclass{amsart}
\usepackage{amsmath,amssymb}
\usepackage{amsthm}
\usepackage{graphicx}
\usepackage{color}
\newtheorem{theorem}{Theorem}[section]   

\newtheorem{proposition}[theorem]{Proposition}

\newtheorem{corollary}[theorem]{Corollary}
\theoremstyle{definition}
\newtheorem{definition}[theorem]{Definition}
\newtheorem{example}[theorem]{Example}
\newtheorem{remark}[theorem]{Remark}
\newtheorem*{remark*}{Remark}

\newcommand{\R}{\mathbb{R}}

 \makeatletter
    
    \@addtoreset{equation}{section}
  \makeatother

\begin{document}

\title{Weak rectangular diagrams, multi-crossing number, and arc index}

\author[T.Ito]{Tetsuya Ito}
\address{Department of Mathematics, Kyoto University, Kyoto 606-8502, JAPAN}
\email{tetitoh@math.kyoto-u.ac.jp}

\begin{abstract}
For a non-split multi-crossing diagram $D$ of a link $L$ we show that $\alpha(L)-2 \leq c_2(D) + \sum_{n> 2}(2n-4)c_n(D)$ holds. Here $\alpha(L)$ is the arc index and $c_n(D)$ is the number of $n$-crossings of $D$. This generalizes and subsumes many known inequalities related to multi-crossing numbers. In the course of proof, we introduce a notion of 
weak rectangular diagram and show that a weak rectangular diagram can be converted to usual rectangular diagram preserving its arc index.
\end{abstract}

\maketitle

\section{Introduction}

A \emph{multi-crossing diagram} $D \subset \R^{2}$ of a (unoriented) link $L$ is an image $p(L)$ of the projection $p:\R^{3} \rightarrow \R^{2}$ such that each of its singular point is an $n$-crossing point for some $n$, where $n$ depends on a singular point. Here an \emph{$n$-crossing point} is a multiplicity $n$ transverse singularity, which is formed by $n$ arc segments intersecting at one point so that any two of them form a transverse double point singularity.

For a multi-crossing diagram $D$ and $n\geq 2$, we denote by $c_n(D)$ the number of its $n$-crossing points. The \emph{total crossing number} $c_{tot}(D)=\sum_{n\geq 2}c_n(D)$ is the number of singular points (crossing points) of $D$.

We say that a multi-crossing diagram is an \emph{$n$-crossing diagram} if all the singular points are $n$-crossing points. The \emph{$n$-crossing number} $c_n(L)$ of a link $L$ is the minimum number of $n$-crossing points among all $n$-crossing diagrams of $L$. The $n$-crossing number was introduced in \cite{ad1}, where it was shown that every link admits an $n$-crossing diagram. The $2$-crossing number $c_2(L)$ is usually called the (minimum) \emph{crossing number} and denoted by $c(L)$. This is one of the most elementary, but theoretically hard to treat invariant for which there are many open problems.

The aim of this paper is to prove the following inequality. Here we say that a multi-crossing diagram is \emph{non-split} if it is connected as a subset of $\R^{2}$.

\begin{theorem}
\label{theorem:main}
For a link $L$ and its non-split multi-crossing diagram $D$, 
\[\alpha(L)-2 \leq c_2(D) + \sum_{n>2}(2n-4) c_n(D)\]
holds. In particular, for a non-split link $L$ and $n>2$,
\[ \alpha(L) -2 \leq (2n-4)c_n(L)\]
holds.
\end{theorem}
Here $\alpha(L)$ is the arc index of $L$. We review the definition in detail in Section \ref{section:arc-index}.

This theorem subsumes or generalizes many known inequalities.
By the Morton-Beltrami inequality \cite{mb} $\mathrm{Breadth}_a F_L(a,z) \leq \alpha(L)-2$ which relates arc index and the Kauffman polynomial $F_L(a,z)$, we get the following.
\begin{corollary}
\label{cor:main}
For a non-split link $L$ and $n>2$
\[\mathrm{Breadth}_a F_L(a,z) \leq (2n-4)c_n(L) \]
holds. 
\end{corollary}

Since $\mathrm{Breadth}_a F_L(a,z) = c(L)$ for non-split alternating links \cite{th}, we get the following fundamental result of arc index.

\begin{corollary}\cite{bp,noy}
\label{cor:arc-index-alternating}
For a non-split alternating link $L$, $c(L)=\alpha(L)-2$.
\end{corollary}

Note that $c(L) \leq \frac{n(n-1)}{2}c_n(L)$ holds, since an $n$-crossing point can be decomposed with a $2$-crossing diagram with $\frac{n(n-1)}{2}$ $2$-crossings. By Corollary \ref{cor:main} and Corollary \ref{cor:arc-index-alternating}, we get the following better bound of $n$-crossing numbers for alternating links.

\begin{corollary}
\label{cor:alternating2}
For a non-split alternating link $L$ and $n>2$, $c(L) \leq (2n-4)c_n(L)$.
\end{corollary}

This subsumes the inequality $\frac{1}{2}c(K) \leq c_3(K)$ for the case $n=3$  \cite{ad1} and $\frac{1}{4}c(K) \leq c_4(K)$ for the case $n=4$ \cite{ad2}.
The proof of these inequalities in \cite{ad1,ad2} are based on a computation of Kauffman bracket. In \cite{ad4}, the Kauffman bracket bound is proven for general $n$, where it is shown that 
\begin{equation}
\label{eqn:Adams-bound} \mathrm{Breadth}\, \langle L \rangle \leq \left(\left\lfloor \frac{n^2}{2} \right\rfloor +4n-8\right)c_n(L)
\end{equation}
holds for $n \geq 3$. Here $\langle L \rangle$ denotes the Kauffman bracket of a (framed) link $L$. Since $\mathrm{Breadth}\,\langle L \rangle = 4c(L)$ for a non-split alternating link, this implies that
\begin{equation}
\label{eqn:Adams-bound2} c(L) \leq \frac{1}{4}\left(\left\lfloor \frac{n^2}{2} \right\rfloor +4n-8\right)c_n(L) 
\end{equation}
holds for an alternating non-split link and $n>2$. 
Although Corollary \ref{cor:main} and the inequality \eqref{eqn:Adams-bound}
are independent, our Corollary \ref{cor:alternating2} gives an improvement of \eqref{eqn:Adams-bound2}.

For the braid index $b(L)$ of $L$, since $2b(L) \leq \alpha(L)$ holds \cite{cr}, we get the following.

\begin{corollary}
For a non-split link $L$, $b(L) \leq \frac{c_2(L)}{2}+1$ and $b(L) \leq (n-2)c_n(L)+1$ hold. 
\end{corollary}

This generalizes the inequalities for the case $n=2$ \cite{oh} and the case $n=3$ \cite{ni}.

Finally, the \emph{\"{u}ber crossing number} $\textit{\"u}\,(L)$ of a link $L$ is the minimum $n$ such that $L$ admits an $n$-crossing diagram with $c_n(D)=1$ \cite{ad3}. Theorem \ref{theorem:main} leads to the following.
\begin{corollary}\cite[Corollary 3.5]{ad3}
$\frac{1}{2}(\alpha(L)+2) \leq \text{\"u}\,(L)$ 
\end{corollary}

Our proof is based on a notion of a \emph{weak rectangular diagram}, a slight generalization of a rectangular diagram which is interesting in its own right (see Definition \ref{definition:weak}). In Theorem \ref{theorem:weak} we show that a weak rectangular diagram can be converted to an honest rectangular diagram without changing its arc index, the number of vertical segments. Since weak rectangular diagrams are more flexible, it helps us to estimate the arc index. In particular, weak rectangular diagram are useful for investigating the relations between the arc index and other link invariants, as our Theorem \ref{theorem:main} demonstrates.

\section{weak rectangular diagram and arc index}

\subsection{Three views of arc index}\label{section:arc-index}
We quickly review three mutually related but slightly different expressions of knots and links that appear in the definition of the arc index.
The relations are summarized in Figure \ref{fig:arc-presentation}.

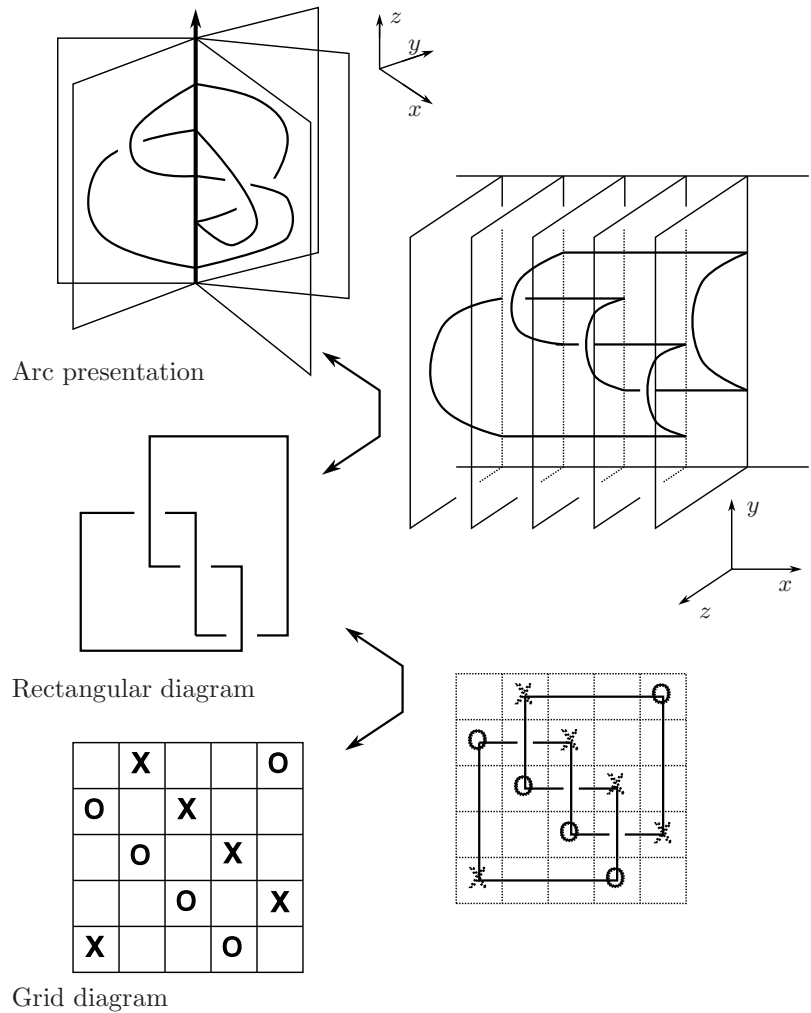
\begin{figure}[htbp]
\begin{center}
\includegraphics*[width=100mm]{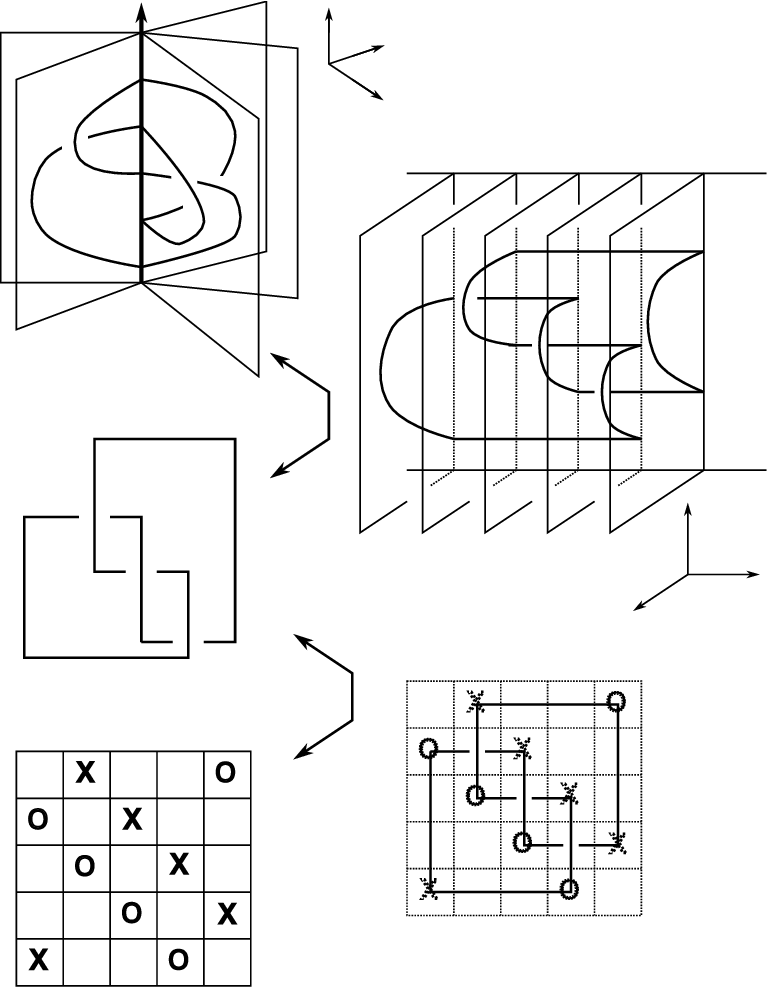}
\begin{picture}(0,0)
\put(-305,225){Arc presentation}
\put(-305,105){Rectangular diagram}
\put(-305,-12){Grid diagram}
\put(-155,325){\small $x$}
\put(-155,350){\small $y$}
\put(-162,360){\small $z$}
\put(-15,145){\small $x$}
\put(-27,175){\small $y$}
\put(-45,135){\small $z$}
\end{picture}
\medskip \medskip
\end{center}
\caption{Arc presentation, rectangular diagram and grid diagram} 
\label{fig:arc-presentation}
\end{figure} 

Although it is a striking feature that these actually represent \emph{Legendrian links}, we do not discuss this issue. We refer to \cite{nt} for the relation to contact geometry.

\subsubsection{Arc presentation}

Take a cylindrical coordinate $(r,\theta,z)$ of $\R^{3}$ which we regard as an 
open book decomposition of $\R^{3}$. By a \emph{page} at $\theta \in [-\pi,\pi)$, we  mean the half-plane $H_{\theta} =\{(r\cos\theta,r\sin \theta,z) \in \R^{3} \: | \: r\geq 0, z \in \R\}$, and by the \emph{binding} we mean the $z$-axis, viewed as the boundary of a page $H_{\theta}$.

\begin{definition}[Arc presentation]
An \emph{arc presentation} $\mathcal{A}$ of a link $L$ is a decomposition of $L$ into simple arcs in a page $H_{\theta}$ connecting two different points of the binding, so that each page $H_{\theta}$ contains at most one arc. We denote by $\alpha(\mathcal{A})$ the number of arcs of arc presentation $\mathcal{A}$ and call it the \emph{arc number}.
\end{definition}

An arc presentation appeared in \cite{bm} in their study of braid index of satellite links, and its fundamental properties are studied in \cite{cr}. 

\subsubsection{Rectangular diagram}

An arc presentation can be illustrated in a slightly different manner, as we will explain shortly.

\begin{definition}[Rectangular diagram]
\label{definition:rectangular}
A \emph{rectangular diagram} $R$ is a (piecewise linear) link diagram that consists of horizontal arcs and vertical arcs such that vertical arcs always appear as an over arc at each crossing points. We denote by $\alpha(R)$ the number of vertical segments of $R$.
\end{definition}

A rectangular diagram is useful for algorithmic treatment of knots and links as discussed in \cite{dy}. Furthermore, a rectangular diagram can be used to distinguish Legendrian links by a geometric method \cite{dp}.

\subsubsection{Grid diagram}

Finally, in a combinatorial reconstruction of knot Fleor homology called grid homology theory \cite{most,oss}, the following more combinatorial presentation of rectangular diagram is used.
 
\begin{definition}[Grid diagram]
A \emph{grid diagram} $\mathbb{G}$ of grid number $n=\alpha(\mathbb{G})$ is an $n \times n$ square grids such that there are $n$ squares marked by the symbol $X$ and $n$ squares marked by the symbol $O$, such that 
\begin{itemize}
\item No square is marked by both $X$ and $O$.
\item Each row contains exactly one $X$ and exactly one $O$.
\item Each column contains exactly one $X$ and exactly one $O$.
\end{itemize}
\end{definition}

For a grid diagram $\mathbb{G}$ one assigns a rectangular diagram $R_{\mathbb{G}}$ by taking a vertical and horizontal lines connecting $X$ and $O$ markings, so that vertical arc always lies above of horizontal arcs. Furthermore, we may naturally assign the orientation by the rule that horizontal arcs are oriented from $O$ to $X$ and vertical arcs are oriented from $X$ to $O$, although in this paper we do not take into account of orientations.

Thus a grid diagram is naturally regarded as a rectangular diagram. Conversely, by adjusting the coordinate of vertical and horizontal segments, a rectangular diagram is easily converted to the form $R_{\mathbb{G}}$. Hence, grid diagrams and rectangular diagrams are equivalent notions.

However, it is important to note that in grid homology theory, its $n\times n$ grid plays an important role so the grid diagram has its own meaning and importance.

\subsection{From arc presentation to rectangular diagram and back}
\label{section:Arc-to-rectangle}

Let $\mathcal{A}$ be an arc presentation of a link $L$ with arc number $n$.
By `blowing-up' the axis, the arc presentation can be converted to a rectangular diagram with $n$ vertical and horizontal arcs.

More precisely, we assume that for the arc presentation $\mathcal{A}$, the page $H_\pi$ contains no arcs hence we view $\mathcal{A}$ as a union of open arcs in $\textrm{Int}(H_\theta) \subset \R^{3} \setminus H_{\pi}$.
Let $f$ be a homeomorphism 
\[ f: \R^{3} \setminus H_\pi \rightarrow (-\pi,\pi)\times \R \times \R_{>0} ,\quad f(r\cos\theta, r \sin \theta,z)=(\theta, z,r).\]
Then $f$ sends an open arc in the page $H_{\theta}$ to an open arc in the open half-plane $\{(x,y,z) \in \R^{3} \: | \: x= \theta, z>0 \}$.

By connecting endpoints of open arcs by the horizontal lines in the $xy$-plane $H=\{(x,y,z) \in \R^{3} \: | \: z=0 \}$, we get a link $L_{\mathcal{A}}$ in $(0,2\pi)\times \R \times \R_{>0} \subset \R^{3}$ isotopic to the link $L$.
By construction, the projection of $L_{\mathcal{A}}$ to the $xy$-plane gives the rectangular diagram with $n$ vertical arcs.

Conversely, for a rectangular diagram $R$ of a link $L$, we put $L$ so that horizontal arcs sit on the $xy$-plane $H=\{(x,y,z) \in \R^{3} \: | \: z=0 \}$, and that each vertical arc $\gamma$ sits on the half-plane $H_{c}=\{(x,y,z) \in \R^{3} \: | \: x=c, z\geq 0 \}$ for some $c \in (-\pi,\pi)$.
Here by abuse of notation, by `horizontal/vertical' arc of a link $L$, we mean a part of link that projects to the horizontal/vertical arc in the rectangular diagram. 

Then by collapsing the $xy$-plane $H$ into the binding, we get the arc presentation. More precisely, by sending the link $L$ by
\[ f^{-1}: (-\pi,\pi)\times \R \times \R_{>0} \rightarrow  \R^{3} \setminus H_\pi \subset \R^{3} \]
we get an arc presentation of the link $L$.

Summarizing, we have three equivalent, but slightly different expressions of links; arc presentation, rectangular diagram, and grid diagram. The arc index is the minimum complexity of such a presentation.

\begin{definition}[Arc index]
The arc index $\alpha(L)$ is defined by
\begin{align*}
\alpha(L) & = \min \{\alpha(\mathcal{A}) \: | \: \mathcal{A} \textnormal{ is an arc presentation } L \} \\
&= \min\{ \alpha(R) \:| \: R  \textnormal{ is a rectangular diagram of } L\}\\
&= \min\{\alpha(\mathbb{G}) \: | \: \mathbb{G}  \textnormal{ is a grid diagram of } L\}
\end{align*}
\end{definition}

\subsection{weak rectangular diagram}

For rectangular diagrams (Definition \ref{definition:rectangular}) it is required that vertical segments always lies above of horizontal segments.
We weaken this requirement to introduce the following.

\begin{definition}[weak rectangular diagram]
\label{definition:weak}
A link diagram $D$ is a \emph{weak rectangular diagram} if it consists of horizontal and vertical arcs such that, each vertical arc contains only over-crossings or under-crossings.
\end{definition}

As in the rectangular diagram, we denote by $\alpha(R)$ the number of vertical arc segments of the weak rectangular diagram $R$. 
A horizontal arc of a weak rectangular diagram may contain both over-crossing and under-crossings (see Figure \ref{fig:weak} (i)).

For a weak rectangular diagram, we say that a vertical segment $\gamma$ is an \emph{over} (resp. \emph{under}) segment if $\gamma$ contains only over-crossings (resp. under-crossings). If $e$ contains no crossing, by convention, we regard $e$ as an under segment.

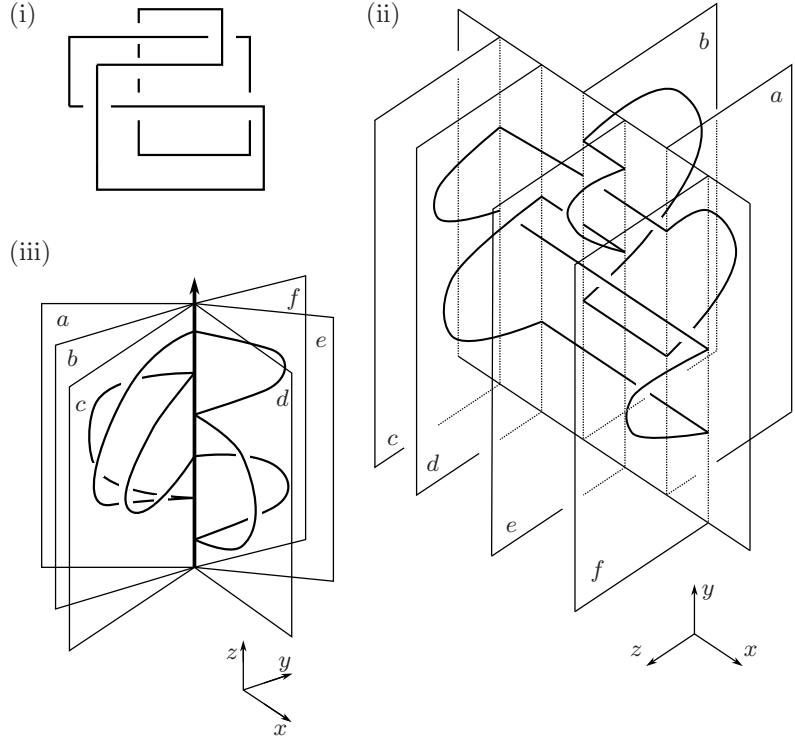
\begin{figure}[htbp]
\begin{center}
\includegraphics*[width=100mm]{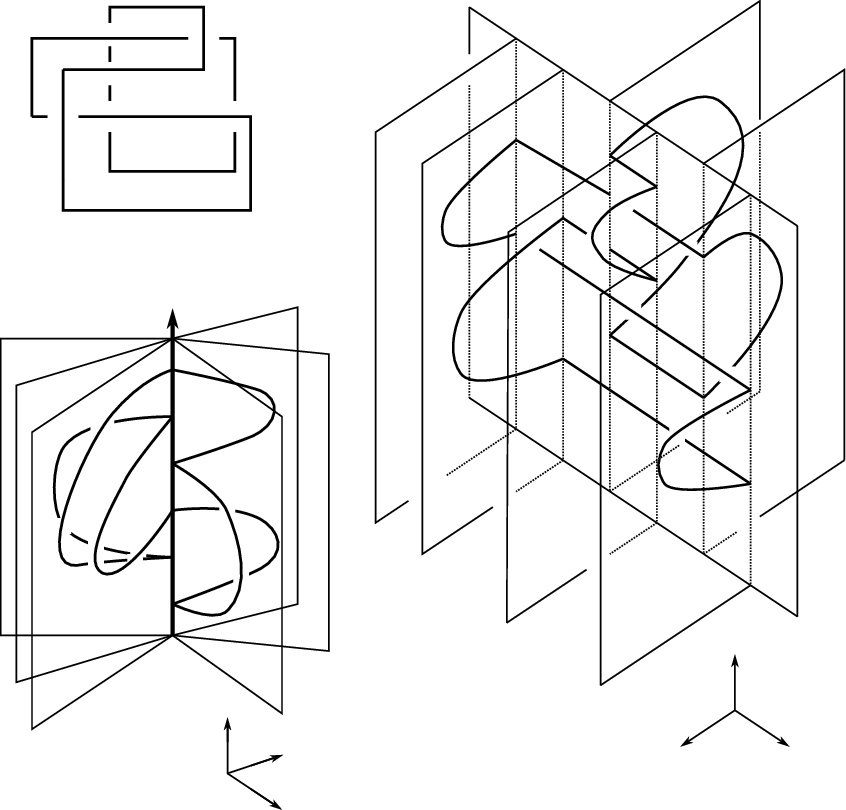}
\begin{picture}(0,0)
\put(-300,265){(i)}
\put(-165,265){(ii)}
\put(-300,175){(iii)}
\put(-200,-5){\small $x$}
\put(-198,21){\small $y$}
\put(-218,24){\small $z$}
\put(-22,25){\small $x$}
\put(-38,48){\small $y$}
\put(-65,25){\small $z$}
\put(-12,235){\small $a$}
\put(-39,255){\small $b$}
\put(-157,105){\small $c$}
\put(-142,95){\small $d$}
\put(-112,72){\small $e$}
\put(-80,55){\small $f$}
\put(-282,150){\small $a$}
\put(-278,135){\small $b$}
\put(-275,118){\small $c$}
\put(-199,120){\small $d$}
\put(-184,142){\small $e$}
\put(-195,158){\small $f$}
\end{picture}
\medskip \medskip
\end{center}
\caption{weak rectangular diagram and arc presentation. (i) weak rectangular diagram. (ii) Three-dimensional illustration of weak rectangular diagram. (iii) Arc presentation. Here $a,b,c,d,e,f$ indicates the correspondence of the pages (half planes)} 
\label{fig:weak}
\end{figure} 

Our crucial but easy observation is that, for a weak rectangular diagram, collapsing $xy$-plane also yields an arc presentation. With a little modification, the construction given in Section \ref{section:Arc-to-rectangle} works for weak rectangular diagrams.

\begin{theorem}
\label{theorem:weak}
If a link $L$ admits a weak rectangular diagram $D$ having $n$ horizontal segments, then $D$ admits an arc presentation with arc number $n$.
\end{theorem}
\begin{proof}
As in Section \ref{section:Arc-to-rectangle}, for a link $L$ with weak rectangular diagram $D$ by abuse of notation, by horizontal/vertical arc of a link $L$, we mean a part of link that projects to a horizontal/vertical arc in the rectangular diagram. 

We put $L$ so that each horizontal arc sits on the $xy$-plane $H=\{(x,y,z) \in \R^{3} \: | \: z=0 \}$.
Furthermore, we put each over vertical arc so that it sits on the upper half-plane $H^{+}_{\theta}=\{(x,y,z) \in \R^{3} \: | \: x=\theta, z\geq 0 \}$ and put each under vertical arcs so that it sits on the lower half-plane $H^{-}_{\theta}=\{(x,y,z) \in \R^{3} \: | \: x=\theta, z\geq 0 \}$ for some $\theta \in (0,\pi)$ (see Figure \ref{fig:weak} (ii)). With no loss of generality, we may assume that $H^{+}_{\theta} \cup H^{-}_{\theta}$ contains at most one vertical arc segment.

Then as in the case of rectangular diagram, by collapsing the $xy$-plane $H$ into the binding, we get an arc presentation. More precisely, let $F$ be the homeomorphism
\[ F: (0,\pi) \times \R \times (\R_{>0} \cup \R_{<0}) \rightarrow ((-\pi,0) \cup (0,\pi)) \times \R \times \R_{>0}, \quad F(\theta,z,r) = ((\mathsf{sgn}\,z)\theta,z,|r|)\]
Then $f \circ F (L)$ gives rise to an arc presentation. (see Figure \ref{fig:weak} (iii))
\end{proof}

In particular, we get the fourth equivalent definition of arc index.

\begin{corollary}
\[ \alpha(L)=\min \{ \alpha(R) \: | \: R \textnormal{ is a weak rectangular diagram of } L \}\]
\end{corollary}

\section{From multi-crossing diagram to weak rectangular diagram}

For non-negative integers $p,q$ with even $p+q$, 
a \emph{$(p,q)$-tangle} is a proper embedding of the disjoint union of $\frac{p+q}{2}$ intervals into $[0,1] \times \R^{2}$ such that $p$ endpoints lie on $\{0\} \times \R^{2}$ and $q$ endpoints lie on $\{1\} \times \R^{2}$.

We remark that in our definition, we do not allow a $(p,q)$-tangle to contain a closed component.

The notion of multi-crossing diagram, rectangular diagram and weak rectangular diagram of tangles are defined in a similar manner. We denote by $\alpha(R)$ the number of vertical segments of a weak rectangular diagram $R$ of a $(p,q)$-tangle diagram.

\begin{definition}
An \emph{elementary $n$-crossing $(p,q)$-tangle} is a connected $(p,q)$-tangle that admits a multi-crossing diagram having exactly one $n$-crossing. We call such a diagram an \emph{elementary tangle diagram} (see Figure \ref{fig:slice-position} (i)).
\end{definition}

\begin{definition}[Sliced position]
\label{definition:sliced-position}
We say that a multi-crossing diagram $D$ is in a sliced position if it satisfies the following conditions (see Figure \ref{fig:slice-position} (ii)).

\begin{itemize}
\item[(a)] the $x$-coordinates of crossing points are $\{\frac{1}{2},\frac{3}{2},\ldots,c-\frac{1}{2}\}$, where $c=c_{tot}(D)$.
\item[(b)] For each $i=1,\ldots,c$, the slice $D_i=D \cap \{(x,y) \in \R^{2} \: | \: i-1 \leq x \leq i\}$ is a tangle diagram which is a disjoint union of the trivial $(1,1)$-tangles and an elementary tangle diagram.
\item[(c)] The trivial $(1,1)$-tangles in each $D_{i}$ are horizontal lines.
\end{itemize}
We call $D_i$ the \emph{($i$-th) slice} of a multi-crossing diagram $D$ in a sliced position.
\end{definition}

Every multi-crossing diagram is put into a sliced position as follows.
First we put $D$ so that the $x$-coordinates of its crossing points satisfy the condition (a). Then we move $D$ so that the geometric intersection number of $D$ and the lines $x=i$ $(i=0,\ldots,c)$ are minimum to achieve (b). Finally, by adjusting the diagram, we achieve (c).

\begin{definition}[Type $0$ slice]
We say that the slice $D_i$ is \emph{of type $0$} if the elementary tangle in $D_i$ is $(0,q)$-tangle or the $(p,0)$-tangle.
\end{definition}

\begin{figure}[htbp]
\begin{center}
\includegraphics*[width=90mm]{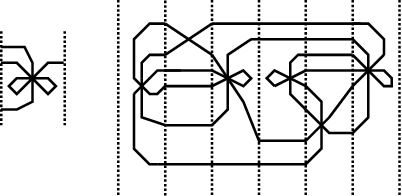}
\begin{picture}(0,0)
\put(-280,120){(i)}
\put(-202,120){(ii)}
\put(-175,0){$D_1$}
\put(-143,0){$D_2$}
\put(-115,0){$D_3$}
\put(-85,0){$D_4$}
\put(-55,0){$D_5$}
\put(-25,0){$D_6$}
\end{picture}
\end{center}
\caption{(i) Elementary $4$-crossing $(3,1)$-tangle and its elementary diagram (ii) Multi-crossing diagram in a sliced position (the height information of multi-crossings are omitted). The slice $D_1,D_4,D_6$ are of type $0$.} 
\label{fig:slice-position}
\end{figure}
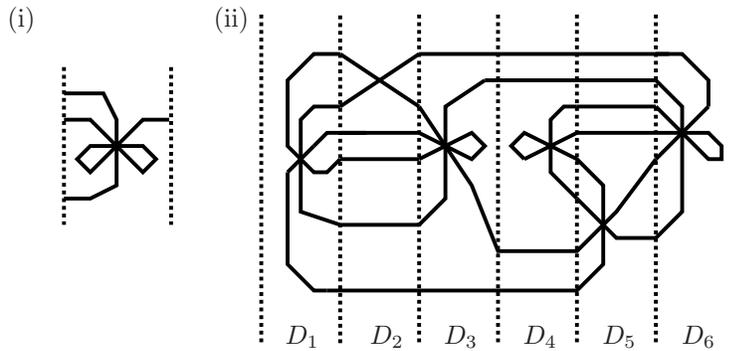

Clearly, $D_1$ and $D_{c}$ must be of type $0$. We show that one can adjust the sliced position so that the other slices are not of type $0$.

\begin{proposition}(c.f. \cite[Theorem 1]{noy})
\label{prop:slice-reduced}
If a multi-crossing diagram $D$ is non-split, then it admits a sliced position such that its $i$-th slice $D_i$ is not of type $0$ for $i \neq 1,c_{tot}(D)$. 
\end{proposition}
\begin{proof}
Assume that the $i$-th slice $D_i$ is of type $0$, and that the elementary tangle $T$ in $D_i$ is a $(0,q)$-tangle for $i \neq 1, c_{tot}(D)$. The case where the elementary tangle is $(p,0)$-tangle case is similar.

We show that by suitably permuting the order of slices (i.e., swapping the $x$-coordinates of crossings), we are able to convert $D_i$ to a slice which is no longer of type $0$, without introducing any slices of type $0$. 

Since $i \neq 1$ and $D$ is non-split, the slice $D_i$ also contains at least one (actually, at least two) trivial $(1,1)$-tangle(s).
The non-splitness of $D$ also implies that there is $i\leq j \leq c$ such that in the tangle diagram $D_{[i,j]}:=D_i \cup D_{i+1} \cup \cdots \cup D_{j}$, the connected component that contains $T$ also contains a trivial $(1,1)$-tangle $e$ in $D_{i}$. Among such $j$, we take the smallest one. 

By changing the order of slices (i.e. by pushing the crossing to the left-hand side if necessary), we may assume that in the tangle diagram $D_{[i,j]}$, the strand of the $(1,1)$-tangle $e$ does not encounter crossings until $D_{j}$ (see Figure \ref{fig:slice-position-reduced}).

Let $\overline{R}$ be the closure of a connected component of $([i,j] \times \R) \setminus D_{[i,j]}$ that contains both the $n$-crossing point of $D_i$ and the (1,1)-tangle $e$. We put $R = \overline{R} \cap [i+1,j] \times \R$.
 
Let $c_{1},\ldots, c_{m}$ be the crossing points that lie on $\partial R$. 
We assume that $c_i$ lies in the slice $D_{k_i}$. We put $\{i,i+1,\ldots,j \} = \{k_1,k_2,\ldots,k_m,\ell_1,\ell_2,\ldots,\ell_{j-i-m}\}$, where
\[ k_1 < k_2 < \cdots< k_m=j, i= \ell_1<\ell_2 < \cdots < \ell_{j-i-m}. \]
Then we reorder the slices $D_{i},D_{i+1},\ldots,D_{j}$ as
\[ D_{k_{m}} < D_{k_{m-1}} < \cdots < D_{k_1} < D_{\ell_1} < \cdots  < D_{\ell_{j-i-m}}.\]
This operation removes type $0$ slice $D_i$ without introducing new type $0$ slices (see Figure \ref{fig:slice-position-reduced}).

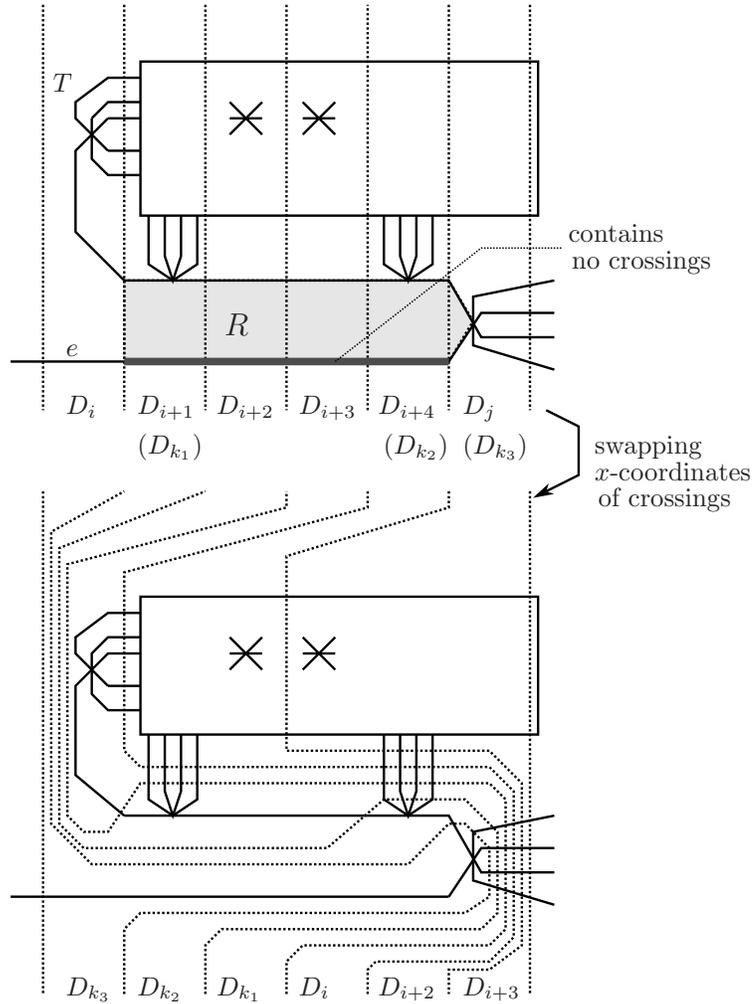
\begin{figure}[htbp]
\begin{center}
\includegraphics*[width=80mm]{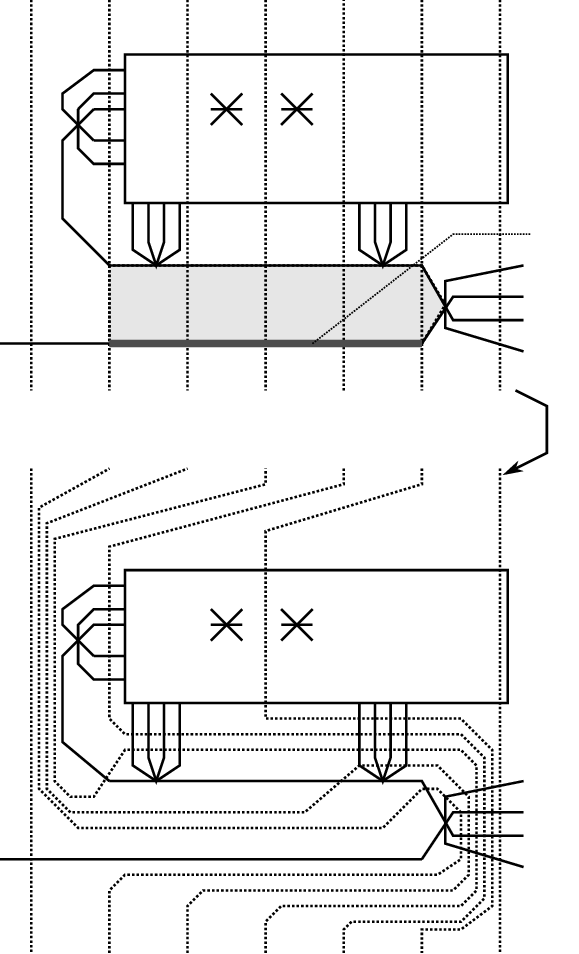}
\begin{picture}(0,0)
\put(-210,242){$e$}
\put(-215,342){$T$}
\put(-150,250){\Large $R$}
\put(-20,285){contains}
\put(-22,275){ no crossings}
\put(-10,205){swapping}
\put(-10,195){$x$-coordinates}
\put(-12,185){ of crossings}
\put(-210,220){$D_i$}
\put(-183,220){$D_{i+1}$}
\put(-183,205){($D_{k_1})$}
\put(-153,220){$D_{i+2}$}
\put(-122,220){$D_{i+3}$}
\put(-92,220){$D_{i+4}$}
\put(-90,205){($D_{k_2})$}
\put(-60,220){$D_{j}$}
\put(-60,205){($D_{k_3})$}
\put(-210,0){$D_{k_3}$}
\put(-183,0){$D_{k_2}$}
\put(-153,0){$D_{k_1}$}
\put(-122,0){$D_{i}$}
\put(-92,0){$D_{i+2}$}
\put(-60,0){$D_{i+3}$}
\end{picture}
\end{center}
\caption{Removing type $0$ slices. We illustrate the effect of changing the $x$-coordinates by drawing how the line $x=a$ changes.} 
\label{fig:slice-position-reduced}
\end{figure} 
\end{proof}

\begin{example}
The proof of Proposition \ref{prop:slice-reduced} gives an algorithm to convert a sliced position into a slice position without non-extremal type $0$ slices 

The mutli-crossing diagram in Figure \ref{fig:slice-position} (ii) have a non-extremal type $0$ slice $D_4$. Following the proof, we remove this type $0$ slice as shown in Figure \ref{fig:slice-position-example}.
\end{example}

\begin{figure}[htbp]
\begin{center}
\includegraphics*[width=105mm]{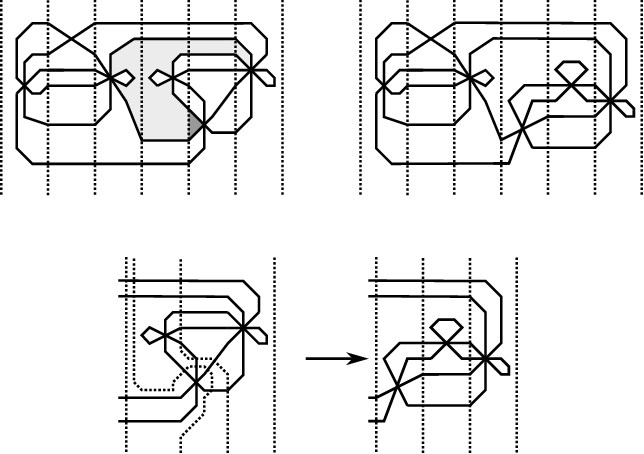}
\begin{picture}(0,0)
\put(-320,210){(i)}
\put(-160,210){(ii)}
\put(-265,100){(iii)}
\put(-232,120){$D_4$}
\put(-210,120){$D_5$}
\put(-228,210){\small $x\!=\!4$}
\put(-202,210){\small $x\!=\!5$}
\end{picture}
\end{center}
\caption{Removing type $0$ slices, example. (i) $D_4$ is a type $0$ slice. The gray region represents $\overline{R}$ and the dark gray region represents $R$. (ii) After the modification. (iii) How the modification is done. We interchange the order of $D_4$ and $D_5$ (as in Figure \ref{fig:slice-position-reduced}, we first illustrate how the line $x=4,5$ changes, then we make lines straight to get a sliced position.} 
\label{fig:slice-position-example}
\end{figure}

Motivated by Proposition \ref{prop:slice-reduced}, we look at the complexity of basic building blocks of a sliced position, a slice which is not of type $0$.
 
\begin{definition}
\label{definition:constant}
For an elementary tangle $T$, we denote by $\alpha(T)$ the minimum number of vertical segments of its weak rectangular diagrams. We define
\[ u_n=\max \{\alpha(T)\: |\: T \mbox{ is an elementary }n\mbox{-crossing tangle which is not of type } 0\}\]
\end{definition}

\begin{example}[$u_2=1$]
Since an elementary $2$-crossing $(1,1)$-tangle is trivial, to compute $u_2$, we need to consider $T_1,\ldots,T_6$ of elementary $2$-crossing tangles. By writing these tangles as a weak rectangular diagram, we confirm that $u_2=1$ (see Figure \ref{fig:u2}).

We remark that when we use a regular rectangular diagram to represent the tangle $T_5,T_6$ then they require at least two vertical arcs. Thus it is crucial to use weak rectangular diagram to get a better (tight) bound. Also, it is clear that a type $0$ elementary $2$-tangle $T$ requires at least two vertical arcs so it is again crucial to exclude type $0$ cases.
\end{example}

\begin{figure}[htbp]
\begin{center}
\includegraphics*[width=90mm]{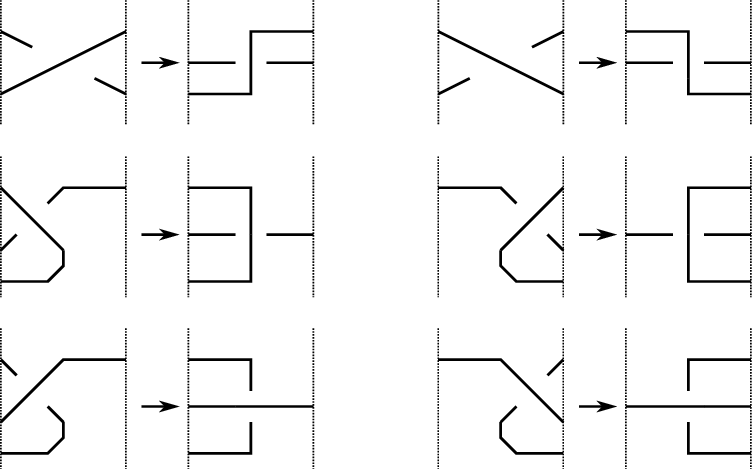}
\begin{picture}(0,0)
\put(-275,135){$T_1$}
\put(-128,135){$T_2$}
\put(-275,75){$T_3$}
\put(-128,75){$T_4$}
\put(-275,15){$T_5$}
\put(-128,15){$T_6$}
\end{picture}
\end{center}
\caption{$u_2=1$. } 
\label{fig:u2}
\end{figure} 

\begin{example}[$u_3=2$]
To determine $u_3$, first we note that an elementary $3$-crossing $(p,q)$-tangle can be written as a $2$-crossing tangle diagram having st most two $2$ crossings if $p+q <6$ (i.e. if some endpoints of strands at the 3-crossing point are connected to form a monogon) (see Figure \ref{fig:u3} (i)). Thus when $p+q \neq 6$ then it is represented by a weak rectangular diagram with at most two vertical arcs. 

For an elementary $3$-crossing $(p,q)$-tangle $T$ with $p+q=6$ (and $p,q \neq 0$), since $p,q \neq 0$, there is at least one strand connecting the left side and the right side. If the height of this strand is the middle, then one can check that $T$ is represented by a weak rectangular diagram with two vertical arcs  (see Figure \ref{fig:u3} (ii)).

Finally, the case where the middle-height strand connects the same side since $p,q \neq 0$, one can check that $T$ is represented by a weak rectangular diagram with two vertical arcs (see Figure \ref{fig:u3} (iii)).

Thus we conclude $u_3=2$.
\end{example}

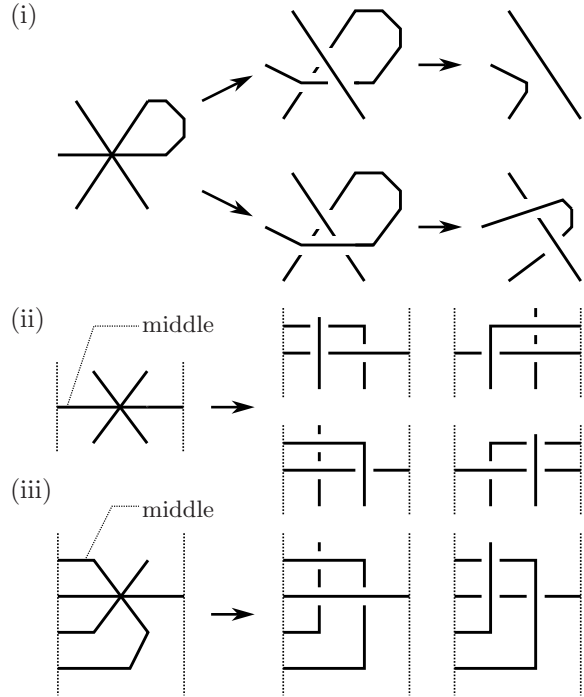
\begin{figure}[htbp]
\begin{center}
\includegraphics*[width=70mm]{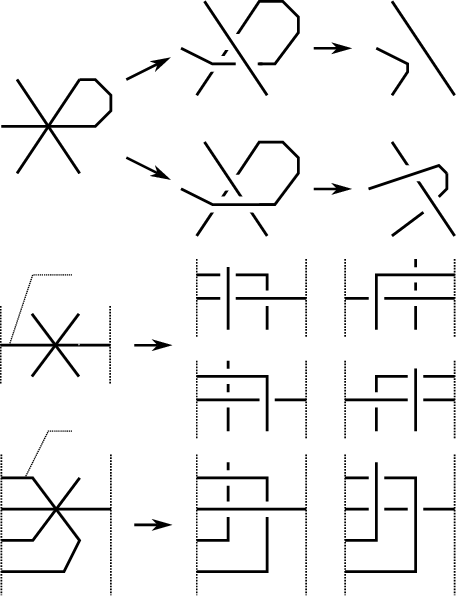}
\begin{picture}(0,0)
\put(-220,255){(i)}
\put(-220,140){(ii)}
\put(-170,138){\small middle}
\put(-220,75){(iii)}
\put(-170,68){\small middle}
\end{picture}
\end{center}
\caption{$u_3=2$. (i) $(p,q)$-tangle with $p+q \neq 6$ admits a 3-crossing (usual) diagram with at most two crossings. (ii) The case where the middle-height strand connects two different sides. (iii) The case where the middle-height strand connects the same side.} 
\label{fig:u3}
\end{figure}

Having computed $u_2$ and $u_3$, we now estimate $u_n$ for general $n$.

\begin{proposition}
\label{prop:un-bound}
For $n>2$, $u_n \leq 2n-4$.
\end{proposition}
\begin{proof}
We prove the proposition by induction on $n$. We already mentioned that $u_3=2$.
To see general cases, we view a general elementary $n$-crossing tangle $T$ as obtained by starting with an elementary $(n-1)$-crossing tangle $T'$ and adding the topmost strand, concatenating its endpoint if necessary.

When $T'$ is of type $0$, we get a weak rectangular diagram of $T$ by adding a horizontal line to a weak rectangular diagram of $T'$. An elementary $(n-1)$-crossing tangle $T'$ of type $0$ can be seen as a concatenation of the cap $\subset$ or $\supset$ (or, both) and an elementary $(n-1)$-crossing tangle $T''$ which is not type $0$.
Since caps $\subset$ and $\supset$ are written as a rectangular diagram having one vertical segment, it follows that $T'$, hence $T$, is represented by a weak rectangular diagram with at most $u_{n-1}+2 \leq 2n-4$ vertical arc segments.

Thus we assume that $T'$ is not of type $0$. Let $R'$ be a weak rectangular diagram of $T'$. Then we get a weak rectangular diagram $R$ of $T$ by putting the topmost arc so that it is away from the weak rectangular diagram $R'$ so that the added arc has at most two vertical segments. Thus $T$ is represented by a weak rectangular diagram having at most $u_{n-1}+2 \leq 2n-4$ vertical arc segments (see Figure \ref{fig:un}).

\end{proof}
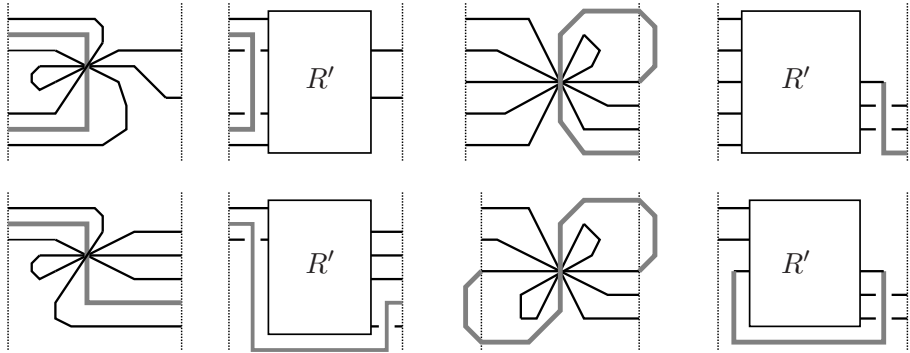
\begin{figure}[htbp]
\begin{center}
\includegraphics*[width=120mm]{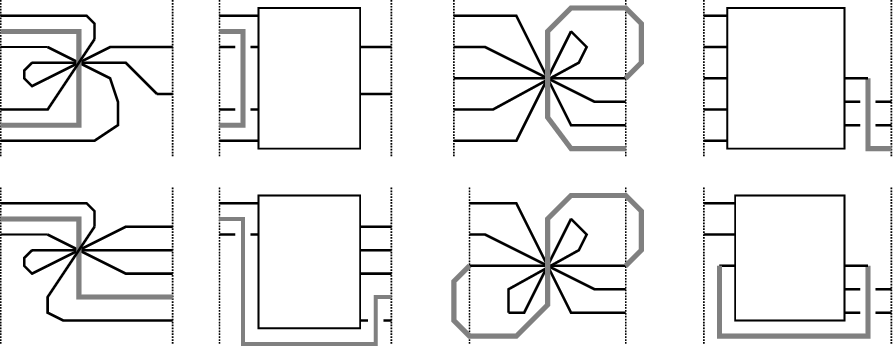}
\begin{picture}(0,0)
\put(-232,100){\large $R'$}
\put(-232,30){\large $R'$}
\put(-52,100){\large $R'$}
\put(-52,30){\large $R'$}
\end{picture}
\end{center}
\caption{The topmost strand (gray arc) can be added so that it is disjoint from a weak rentangular diagram $R'$ and that it has at most two vertical segments.} 
\label{fig:un}
\end{figure}

%

\begin{proof}[Proof of Theorem \ref{theorem:main}]
By Proposition \ref{prop:slice-reduced} we put $D$ so that each slice $D_i$ is not of type $0$, except $D_1$ and $D_{c_{tot}(D)}$.
We view $D_1$ and $D_{c_{tot}(D)}$ as a concatenation of the caps $\subset$, $\supset$ and elementary tangles which are not of type $0$.
 
Since caps $\subset$ and $\supset$ are written as a rectangular diagram having one vertical segment, we conclude that $L$ is written as a weak rectangular diagram having at most $2+\sum_{n\geq 2}c_n(D)u_n$ vertical arcs. Therefore by Theorem \ref{theorem:weak} and Proposition \ref{prop:un-bound}, $\alpha(L)-2 \leq c_2(D)+ \sum_{n>2}(2n-4)c_n(D)$.
\end{proof}

\begin{remark}
We noted that our argument can be seen as a generalization of arguments in \cite{noy} where they introduced a notion of bisected vertex labeling of plane graphs and gave a proof of the inequality $\alpha(L) \leq c(L)+2$. Our method of changing a weak rectangular diagram into an arc presentation appeared in \cite[Section 4]{noy} as a special case. Also, the notion of bisected labeling in \cite{noy} is essentially the same as a special position which we showed in Proposition \ref{prop:slice-reduced}.
\end{remark}

\section*{Acknowledgements}
The author is partially supported by JSPS KAKENHI Grant Numbers 21H04428 and 23K03110. 
The author would like to thank the referee for a careful reading and valuable comments that improve the exposition.

\end{document}